\let\ORIlabel\label
\let\ORIrefstepcounter\refstepcounter
\AddToHook{package/hyperref/before}{
  \let\label\ORIlabel
  \let\refstepcounter\ORIrefstepcounter
}

\documentclass[final]{siamart220329}

\usepackage{graphicx}%
\usepackage{amsmath,amssymb,amsfonts}%
\usepackage{xcolor}%
\usepackage{algorithm}
\usepackage{algpseudocode}%
\usepackage{hyperref}

\usepackage{mystylefile}
\headers{Preconditioned Normal Equations in Mixed Precision}{J. E. Garrison and I.C.F. Ipsen}
\title{Perturbation Analysis for Preconditioned\\ Normal Equations in Mixed Precision}

\author{James E.\ Garrison\thanks{Department of Mathematics, North Carolina State University,
NC 27695-8205, USA
  (\email{jegarri3@ncsu.edu}).}
\and Ilse C.F.\ Ipsen\thanks{Department of Mathematics, North Carolina State University,
NC 27695-8205, USA
  (\email{ipsen@ncsu.edu}, \url{https://ipsen.math.ncsu.edu/}).}}

\usepackage{amsopn}

\begin{document}
\maketitle

\begin{abstract}
For real matrices of full column-rank, 
we analyze the conditioning of several types of normal equations that are preconditioned by a randomized preconditioner computed in lower precision. These include 
symmetrically preconditioned normal equations,
half-preconditioned normal equations, seminormal equations and not-normal equations.
 Our perturbation bounds are realistic and informative, and suggest that the conditioning depends only mildly on the quality of the preconditioner; however, it
does depend on the size of the least squares residual
-- even if the normal equations do not originate from a least squares problem.
We illustrate that a randomized preconditioner can deliver
a solution accuracy comparable to that of  Matlab’s mldivide command, is efficient in practice, and 
well-suited to GPU implementations. For the computation 
of the preconditioner,
we propose an automatic selection of the precision, based on a fast condition
number estimation in lower precision.
\end{abstract}

\begin{MSCcodes}
65F45, 65F20, 15A12
\end{MSCcodes}

\section{Introduction}

Given a matrix $\ma\in\rmn$ with $\rank(\ma)=n$, and $\vb\in\real^m$, we consider
the conditioning of the normal equations 
\begin{equation}\label{e_ne1}
\ma^T\ma\vx=\ma^T\vb.
\end{equation}
Although normal equations represent a simple and effective solution of
least squares problems
\begin{equation}\label{e_ls}
\|\ma\vx-\vb\|_2=
\min_{\vx}{\|\ma\vx-\vb\|_2},
\end{equation}
they can be illconditioned and are not recommended in many applications~\cite[Section 5.3.7]{GovL13}. 
For instance, once the condition number of $\ma$ with respect to left inversion exceeds $10^7$,
then $\ma^T\ma$ is numerically singular in IEEE double precision.
Hence, for the normal equations to be well conditioned and safe to use, the condition number of $\ma$ must be small,
ideally close to~1. 
  
We analyze two preconditioning approaches from \cite{ipsen2025solutionsquaresproblemsrandomized}
for improving the conditioning of the normal equations, but now with the preconditioner computed in lower precision.

\paragraph{Preconditioned Normal Equations  (PNE). }
Precondition the matrix $\ma$  with a randomized preconditioner $\mr_s$ so that the preconditioned matrix 
$\ma_p\equiv \ma\mr_s^{-1}$ is well conditioned
with high probability. Then solve the preconditioned normal equations and recover the original solution,
\begin{equation}
\begin{split}
\ma_p^T\ma_p\vy&=\ma_p^T\vb\\\label{e_pne}
\mr_s\vx&=\vy.
\end{split}
\end{equation}

\paragraph{Half Preconditioned Normal Equations (HPNE).}
Dispense with the triangular system solution by preconditioning only the left instances of $\ma$,
\begin{align}\label{e_hpne}
\ma_p^T\ma\vx=\ma_p^T\vb.
\end{align}
Since $\ma_p^T\ma$ is non-symmetric, the linear system 
has to be solved by an LU factorization
with partial pivoting, a QR factorization, or an iterative solver. 
The HPNE~(\ref{e_hpne}) represent a special case of the not-normal equations~\cite{Wathennotnormal},
\begin{equation*}
    \mb^T\ma\vx=\mb^T\vb,
\end{equation*}
where $\mb\in\rmn$ is a full column rank matrix with the same column space as $\ma$. 

The perturbation analysis in~\cite{ipsen2025solutionsquaresproblemsrandomized} suggests that PNE and HPNE, with an effective preconditioner so that $\kappa(\ma_p)\lesssim10$, are as well conditioned as $\ma$, and 
their solution can be almost as accurate as that from the Matlab backslash ($\verb|mldivide|$) command. However, the bounds
from~\cite{ipsen2025solutionsquaresproblemsrandomized}
are not informative if the preconditioner is computed in a lower precision. In contrast, our new bounds for PNE and HPNE in Sections~\ref{s_PNE} and~\ref{s_HPNE} are realistic, as illustrated in Figure~\ref{fig:HPNEbadbound}.


\begin{figure}
    \centering
    \includegraphics[width=0.48\linewidth]{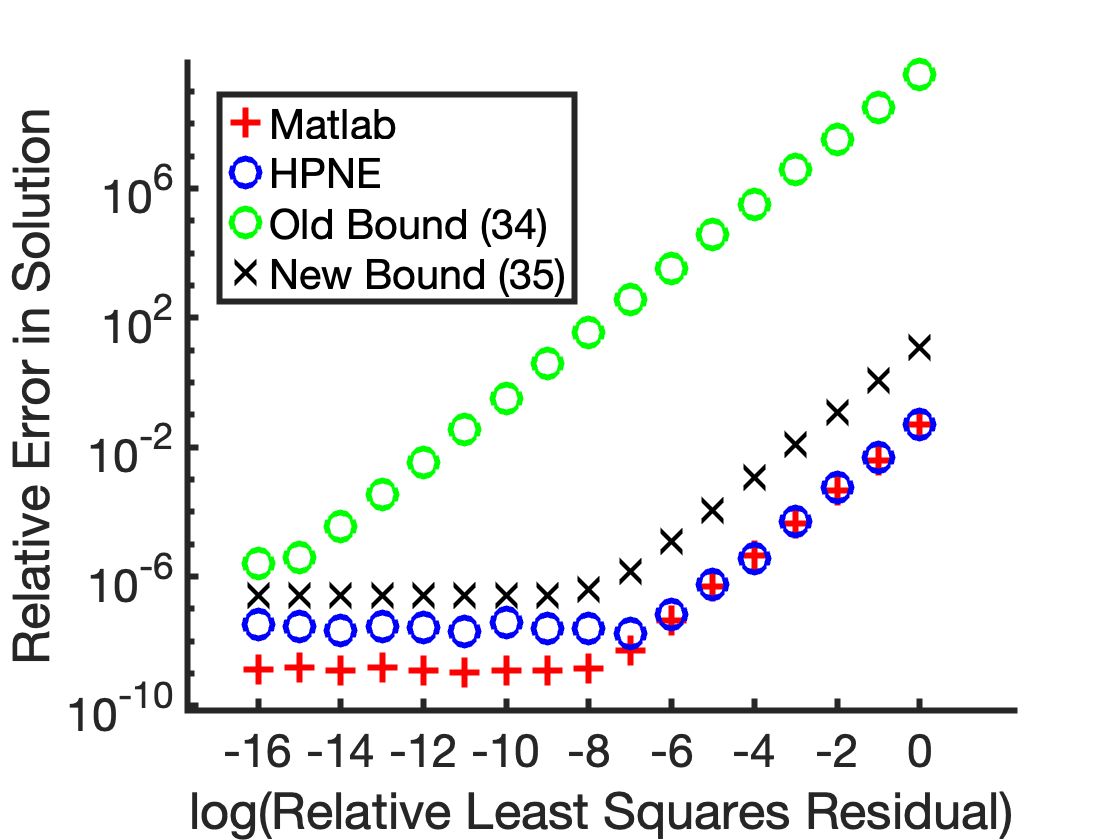}
    \caption{Relative errors $\|\vhx-\vx_*\|/\|\vhx\|$ in the computed solutions $\vhx$ 
and perturbation bounds versus logarithm of relative 
least squares residuals $\|\vb-\ma\vx_*\|/(\|\ma\|\|\vx_*\|)$ for 
$\ma\in\real^{1,000\times 100}$ with condition number $\kappa(\ma)=10^8$.
Shown are 
the Matlab backslash solutions (red plusses); the HPNE solutions with a single precision preconditioner (blue squares); the generalized version 
\eqref{eq42}
of
the perturbation bound \cite[Theorem 3.1]{ipsen2025solutionsquaresproblemsrandomized} (green circles); and our new bound~\eqref{eq32} (black x's). } 
    \label{fig:HPNEbadbound}
\end{figure}

\subsection{Contributions and Overview}

 We present perturbation bounds for PNE and HPNE when the preconditioner is computed in lower precision, and illustrate the potential for faster
 speed.

\begin{enumerate}
\item Our perturbation bounds are realistic and informative for both PNE (Theorem~\ref{t_55}
in Section~\ref{PNE_new}) and HPNE (Theorem~\ref{t_22} in Section~\ref{HPNE_new}), and imply
that their conditioning depends only very weakly on the preconditioner~$\mr_s$.
    \item The mixed precision solution of PNE and HPNE by direct methods on NVIDIA H100 GPUs shows potential for speedups over established direct methods (Section~\ref{exp_HPC}). 
   
    \item Our perturbation bound for the seminormal equations (Section~\ref{s_sne}) implies that they are no better conditioned  than the normal equations.
    \item Like for PNE and HPNE, the perturbation bound for the not-normal equations (Section~\ref{s_nne}) depends on 
    the least squares residual and implies that they are well conditioned when they are close to the HPNE.
\item We propose an automatic selection of the precision level, based on a fast condition
number estimation in the lower precision (Algorithm~\ref{alg_5}
in Section~\ref{s_lpp}).
\end{enumerate}

\paragraph{Overview.} After defining notation (Section~\ref{s_notation}), we review existing work (Section~\ref{s_exist}), followed by perturbation bounds for the PNE (Section~\ref{s_PNE}) and the HPNE (Section~\ref{s_HPNE}), as well as for the seminormal and not-normal equations (Section~\ref{s_var}). We present the randomized preconditioner in lower precision (Section~\ref{s_lpp}) and end with  numerical experiments (Section~\ref{s_num}).

\subsection{Notation}\label{s_notation}
 For a matrix $\ma\in\rmn$ with 
$\rank(\ma)=n$, the Moore-Penrose inverse is $\ma^{\dagger}\equiv(\ma^T\ma)^{-1}\ma^T$, and 
the two-norm condition number with respect to left inversion is 
$\kappa(\ma)\equiv\|\ma\|\|\ma^{\dagger}\|$, where $\|\cdot\|$ denotes the Euclidean two-norm.

To put subsequent bounds in context, we review perturbation bounds for least squares problems~(\ref{e_ls}), and the normal equations~(\ref{e_ne1}).

\begin{lemma}[Fact 5.14 in \cite{IIbook}]\label{l_ls}
Let $\ma, \ma+\me\in\rmn$ with $\rank(\ma)=\rank(\ma+\me)=n$ and 
$\epsilon_A\equiv \|\me\|/\|\ma\|$. Let $\vx_*$ be the solution 
to $\min_{\vx}{\|\ma\vx-\vb\|}$ and $\vhx\neq \vzero$ the solution to $\min_{\vx}{\|(\ma+\me)\vx-\vb\|}$.
Then
\begin{align*}
\frac{\|\vhx-\vx_*\|}{\|\vhx\|}\leq \kappa(\ma)\,\epsilon_A\, 
\left(1+\kappa(\ma)\frac{\|\vb-(\ma+\me)\vhx\|}{\|\ma\|\|\vhx\|}\right).
\end{align*}
\end{lemma}

We do not consider perturbations in the right hand side $\vb$, because perturbations in $\ma$ tend to be more influential on the sensitivity of the least squares problems.

The following bound for the normal equations makes no assumptions on the perturbation $\me$, so that $\ma+\me$ can be rank deficient.

\begin{lemma}[Lemma A.1 in \cite{ipsen2025solutionsquaresproblemsrandomized}]\label{l_ne}
Let $\ma,\ma+\me\in\rmn$ with $\rank(\ma)=n$; 
$\epsilon_A\equiv \|\ma\|/\|\me\|$; and
\begin{align*}
\ma^T\ma\vx_*=\ma^T\vb, \qquad
(\ma+\me)^T(\ma+\me)\vhx=(\ma+\me)^T\vb.
\end{align*}
If $\vhx\neq \vzero$ then
\begin{align*}
\frac{\|\vhx-\vx_*\|}{\|\vhx\|}
\leq \kappa(\ma)^2\epsilon_A\>\left(
\frac{\|\vb-\ma\vhx\|}{\|\|\ma\|\|\vhx\|}+1+\epsilon_A\right).
\end{align*}
\end{lemma}

Lemma~\ref{l_ne} suggests that the accuracy
of the normal equations only depends on the least squares residual when it is too large,
$\frac{\|\vb-\ma\vhx\|}{\|\ma\|\|\vhx\|}>1$.

\section{Existing Work}\label{s_exist}
The use of normal equations is often discouraged due to the 'matrix squaring problem,' where 
$\kappa(\ma^T\ma)=\kappa(\ma)^2$, leading to a potentially much worse conditioned linear system. Preconditioning the normal equations is not obvious, as effective preconditioners for $\ma$ do not necessarily make effective preconditioners for~$\ma^T\ma$~\cite{Wathen2022}. 

Nonsymmetric systems $\ma\vx=\vb$ can be solved
by applying  preconditioned iterative methods such as CGNE to the normal equations
\cite[Section 11.3.9.]{GovL13}, \cite{PT2024}. In~\cite{LNZ2025}, CGNE preconditioners are constructed for
linear systems that arise from certain PDE discretizations.
In \cite{Wathennotnormal}, the iterative solution of the not-normal equations $\mb^T\ma\vx=\mb^T\vb$ is proposed, where $\mb\in\rmn$ has the same column space of $\ma$ and can be obtained from an LU factorization of $\ma$.

The iterative solution of preconditioned normal equations with iterative refinement via backward stable algorithms is
presented in~\cite{epperly2025fastrandomizedleastsquaressolvers,xu2024randomizediterativesolveriterative}. The corresponding
FOSSILS/SPIR and SIRR algorithms have an operation count of $\mathcal{O}( mn+ n^3)$.
 An alternative use of iterative refinement for least squares problems is to iteratively refine the seminormal equations $\mr^T\mr\vx=\ma^T\vb$ \cite{carson2025}, where it is observed  
that the seminormal equations, when combined with iterative refinement, are not sensitive to the size of the relative least squares residual.

Although algorithms like FOSSILS and SIRR are numerically stable, the cost of computing the preconditioner remains an issue. Mixed-precision approaches can reduce this cost. For linear systems, Higham and Pranesh~\cite{highampranesh} propose 
Cholesky factorizations in lower precision as  preconditioners for iterative refinement and extend this to the solution of least squares problems via normal equations. Similarly, Scott and T{\r u}ma~\cite{ScottTuma2025} compute incomplete Cholesky factors in lower precision to precondition LSQR. Li~\cite{Li_2024} also investigates mixed-precision within the LSQR algorithm for solving discrete linear ill-posed problems via regularized least squares.

Beyond low-precision factorizations, many randomized sketching techniques are inherently parallelizable and can benefit from GPU acceleration. Chen et. al.~\cite{chen2025gpuparallelizablerandomizedsketchandpreconditionlinear} give a thorough benchmarking of a GPU-based sketch-and-precondition solver based on a sparse sign embedding. Carson and Daužickaitė~\cite{carson2025mixedprecisionsketchingleastsquares} compute a mixed-precision sketch as the preconditioner for GMRES with iterative refinement. 

\section{Preconditioned Normal Equations (PNE)}\label{s_PNE}
We extend an existing PNE perturbation bound to mixed precision (Section~\ref{PNE_gen}) and then present an improvement (Section~\ref{PNE_new}).

\subsection{Extension of an existing PNE perturbation bound to mixed precision}\label{PNE_gen}
A previous bound \cite[Theorem 2.1]{ipsen2025solutionsquaresproblemsrandomized} 
shows that the conditioning of the PNE depends on the least squares residual of the original least squares problem.
Theorem~\ref{t_3} below is a small extension where  
the preconditioner can be computed in a different precision. 

\begin{theorem}\label{t_3}
Let $\ma\in\rmn$ have $\rank(\ma)=n$; $\mr_s\in\rnn$ be nonsingular; $\me_s\in\rnn$, $\me_A\in\rmn$; 
\begin{align*}
\ma_1\equiv \ma(\mr_s+\me_s)^{-1},\qquad \ma_2\equiv \ma_p+\me_p, \qquad
\epsilon_s\equiv\frac{\|\me_s\|}{\|\mr_s\|}, \qquad \epsilon_p\equiv\frac{\|\me_p\|}{\|\ma_p\|};
\end{align*}
and $\|\me_s\|\|\mr_s\inv\|<1$. The computed solutions corresponding 
to (\ref{e_pne}) are 
\begin{align}
\ma_1^T\ma_2\vhx&=\ma_1^T\vb\label{e_t211}\\
\mr_s\vhx&=\vhy.\label{e_pne22}
\end{align}
If  $\vhx\neq \vzero$ and $\vhy\neq \vzero$, then
\begin{align*}
\frac{\|\vx_*-\vhx\|}{\|\vhx\|}\leq \kappa(\mr_s)\kappa(\ma_p)\,\nu\,
\left(\epsilon_p+\kappa(\ma_p)\, \eta_s
\left(\frac{\|\ma_p\vhy-\vb\|}{\|\ma_p\|\|\vhy\|}+\epsilon_p\right)\right),
\end{align*}
where 
\begin{align*}
\nu\equiv \frac{\|\mr_s\vhx\|}{\|\mr_s\| \|\vhx\|}\leq 1\qquad \text{and}\qquad
\eta_s\equiv \frac{\kappa(\mr_s)\epsilon_s}{1-\kappa(\mr_s)\,\epsilon_s}. 
\end{align*}
\end{theorem}

\begin{proof}
The proof is analogous to that of \cite[Theorem 2.1]{ipsen2025solutionsquaresproblemsrandomized}, but omits
the simplification $\epsilon\equiv\max\{\epsilon_p,\epsilon_s\}$. 
\end{proof}

If the perturbation $\epsilon_s$ in the preconditioner is large compared to $\kappa(\mr_s)\inv$, then the bound in Theorem~\ref{t_3} can be dominated by the least squares residual. Section~\ref{exp_accuracy} illustrates that this makes this bound uninformative in mixed-precision. 

\subsection{Improved PNE perturbation bound}\label{PNE_new}
 We improve Theorem~\ref{t_3} with a bound that is informative in both mixed precision and a single working precision.

\begin{theorem}\label{t_55}
Let $\ma\in\rmn$ have $\rank(\ma)=n$; $\mr_s\in\rnn$ be nonsingular; $\me_s\in\rnn$ and $\me_A\in\rmn$; $\|\me_s\|\|\mr_s^{-1}\|<1$, and
\begin{eqnarray*}
   \ma_1\equiv(\ma+\me_A)(\mr_s+\me_s)^{-1}, \qquad  \epsilon_s\equiv\frac{\|\me_s\|}{\|\mr_s\|},\qquad
   \epsilon_A\equiv\frac{\|\me_A\|}{\|\ma\|}.
\end{eqnarray*}
The computed solutions corresponding to (\ref{e_pne}) are 
\begin{eqnarray}
\ma_1^T\ma_1\vhy&=\ma_1^T\vb\label{e_pne2456}\\
(\mr_s+\me_s)\vhx&=\vhy.\label{e_pne2256}
\end{eqnarray}
If $\vhx\neq \vzero$ and $\vhy\neq\vzero$, then 
\begin{align*}
       \frac{\|\vhx-\vx_*\|}{\|\vhx\|} \leq \kappa(\mr_s)\kappa(\ma_p)\epsilon_A\left(\kappa(\ma_p)\kappa(\mr_s)\frac{\|\ma\vhx-\vb\|}{\|\ma\|\|\vhx\|}+1+\kappa(\ma)\epsilon_A\right).
    \end{align*}
\end{theorem}

\begin{proof}
Rearrange~\eqref{e_pne2456} and insert \eqref{e_pne2256} 
to find the computed least squares residual
\begin{align}\label{eq_1}
        \ma_1^T\vr=-\ma_1^T\me_A\vhx, \qquad \text{where}\qquad
        \vr\equiv \ma\vhx-\vb.
    \end{align}
Since $\|\me_s\|\|\mr_s^{-1}\|<1$, the Banach lemma \cite[Lemma 2.3.3]{GovL13} implies that $\mr_s+\me_s$ is nonsingular, so that $\ma_1$ is well defined. Because $\ma$ has full column rank, 
\begin{align}\ma+\me_A=(\mi+\me_A\ma^\dagger)\ma=(\mi+\mh)\ma, \qquad \text{where}\qquad
\mh\equiv\me_A\ma^\dagger.\label{eqbanach}
\end{align} 
To relate $\ma_1$ to $\ma_p$, we combine~\eqref{eqbanach} with $(\mr_s+\me_s)\inv=\mr_s^{-1}(\mi+\me\mr_s^{-1})^{-1}$,
\begin{align*}
\ma_1=(\mi+\mh)\ma_p(\mi+\me_s\mr_s\inv)\inv.
\end{align*}
Multiply (\ref{eq_1}) by $(\mi+\me_s\mr_s\inv)^T$ on both sides, and apply the above expression of $\ma_1$,
\begin{align*}
    \ma_p^T(\mi+\mh)^T\vr=-\ma_p^T(\mi+\mh)^T\me_A\vhx.
\end{align*}
 Isolate $\ma_p^T\vr$,
 \begin{align*}
    \ma_p^T\vr= -\ma_p^T\mh^T\vr-\ma_p^T(\mi+\mh)^T\me_A\vhx,
\end{align*}
and multiply both sides by $(\ma_p^T\ma_p)\inv$,
\begin{align*}
    \ma_p^\dagger\ma\vhx-\vy_*= -(\ma_p^T\ma_p)\inv\mr_s^{-T}\me_A^T\vr-\ma_p^\dagger(\mi+\mh)^T\me_A\vhx.
\end{align*}
Rewrite the left hand side as
\begin{align*}
     \ma_p^\dagger\ma\vhx-\vy_*=\mr_s\ma^\dagger\ma\vhx-\mr_s\vx_*=\mr_s(\vhx-\vx_*),
\end{align*}
which implies \begin{align}\label{eq_2}
    \vhx-\vx_*=\mr_s\inv\left(-(\ma_p^T\ma_p)\inv\mr_s^{-T}\me_A^T\vr-\ma_p^\dagger(\mi+\mh)^T\me_A\vhx\right).
\end{align}
Bound the norm of $ \vhx-\vx_*$ by bounding the norm of each summand on the right hand side of~\eqref{eq_2} in turn. 
First, \begin{align}\label{eq_3}
    \|\mr_s^{-T}\me_A^T\vr\|\leq\frac{\kappa(\mr_s)\|\me_A\|}{\|\mr_s\|}\|\vr\|.
\end{align}
Abbreviating $\alpha\equiv 1+\kappa(\ma)\epsilon_A$ gives
\begin{align}
   \| \ma_p^\dagger(\mi+\mh)^T\me_A\vhx)\|
   & \leq \alpha\|\ma_p^\dagger\|\|\me_A\|\|\vhx\|. \label{eq_4}
\end{align} 
Combine \eqref{eq_2}, \eqref{eq_3}, and~\eqref{eq_4},
\begin{align*}
    \frac{\|\vhx-\vx_*\|}{\|\vhx\|}\leq \|\mr_s\inv\|\kappa(\ma_p)\left(\frac{\kappa(\ma_p)}{\|\ma_p\|^2}\,\frac{\kappa(\mr_s)\|\me_A\|}{\|\mr_s\|}\frac{\|\vr\|}{\|\vhx\|}+\alpha\frac{\|\me_A\|}{\|\ma_p\|}\right).\
\end{align*}
From $\frac{1}{\|\ma_p\|\|\mr_s\|}\leq\frac{1}{\|\ma\|}$ follows
\begin{align*}
    \frac{\|\vhx-\vx_*\|}{\|\vhx\|}&\leq \|\mr_s\inv\|\kappa(\ma_p)\left(\kappa(\ma_p)\kappa(\mr_s)\epsilon_A\frac{\|\vr\|}{\|\ma_p\|\|\vhx\|}+\alpha\frac{\|\me_A\|}{\|\ma_p\|}\right)\\
&\leq\kappa(\mr_s)\kappa(\ma_p)\epsilon_A\left(\kappa(\ma_p)\kappa(\mr_s)\frac{\|\vr\|}{\|\ma\|\|\vhx\|}+\alpha\right).
\end{align*}
\end{proof}

 Theorem~\ref{t_55} implies that to first order the error is bounded by \begin{align}\label{e_pnefo}
     \frac{\|\vx_*-\vhx\|}{\|\vhx\|}\lesssim
\kappa(\mr_s)\kappa(\ma_p)\epsilon_A\max\left\{\kappa(\ma_p)\kappa(\mr_s)\frac{\|\ma\vhx-\vb\|}{\|\ma\|\|\vhx\|},1\right\}
 \end{align}
With an effective preconditioner, so that $\kappa(\ma_p)\lesssim10$, this bound resembles the one in Lemma~\ref{l_ls} since  $\kappa(\mr_s)\kappa(\ma_p)\approx\kappa(\ma)$. Theorem~\ref{t_55} 
is not limited to triangular matrices but 
holds for any nonsingular preconditioner $\mr_s$.

Unlike Theorem~\ref{t_3}, the bound in Theorem~\ref{t_55} does not depend on the perturbation~$\epsilon_s$ in the preconditioner,
as long as $\mr_s+\me_s$ is nonsingular. Thus, Theorem~\ref{t_55} is more informative than Theorem~\ref{t_3}
if $\epsilon_s$ is large.

\section{Half-Preconditioned Normal Equations (HPNE)} \label{s_HPNE}
We extend an existing HPNE perturbation bound to mixed precision (Section~\ref{HPNE_gen}) and then present an improved bound (Section~\ref{HPNE_new}).

\subsection{Extension an existing HPNE perturbation bound to mixed precision}\label{HPNE_gen}
A previous bound \cite[Theorem 3.1]{ipsen2025solutionsquaresproblemsrandomized} shows that the 
conditioning of the HPNE solution depends on the least squares residual  of the original least squares problem. Theorem~\ref{t_2} below is a small extension where the preconditioner can be computed in a different precision. 

\begin{theorem}\label{t_2}
Let $\ma\in\rmn$ have $\rank(\ma)=n$; $\mr_s\in\rnn$ be nonsingular;  $\me_s\in\rnn$, $\me_A\in\rmn$; 
\begin{align*}
\ma_1\equiv \ma(\mr_s+\me_s)^{-1},\qquad \ma_2\equiv \ma+\me_A, \qquad
\epsilon_s\equiv\frac{\|\me_s\|}{\|\mr_s\|}, \qquad \epsilon_A\equiv\frac{\|\me_A\|}{\|\ma\|},
\end{align*}
and $\|\me_s\|\|\mr_s\inv\|<1$.
The computed solutions corresponding to (\ref{e_hpne}) are
\begin{align}
\ma_1^T\ma_2\vhx=\ma_1^T\vb.\label{e_t21}
\end{align}
If  $\vhx\neq \vzero$, then
\begin{align*}
\frac{\|\vx_*-\vhx\|}{\|\vhx\|}\leq \kappa(\ma_p^T\ma)\, \nu\, 
\left(\eta_s\frac{\|\vb-\ma\vhx\|}{\|\ma\|\|\vhx\|}+(1+\eta_s)\epsilon_A\right),
\end{align*}
where 
\begin{align*}
\nu\equiv \frac{\|\ma_p\|\|\ma\|}{\|\ma_p^T\ma\|}\geq 1,\qquad 
\eta_s\equiv \frac{\kappa(\mr_s)\epsilon_s}{1-\kappa(\mr_s)\epsilon_s}.
\end{align*}
\end{theorem}

\begin{proof}
The proof is analogous to that of \cite[Theorem 3.1]{ipsen2025solutionsquaresproblemsrandomized}, but 
omits the simplification $\epsilon=\max\{\epsilon_A,\epsilon_s\}$. \end{proof}

If the perturbation $\epsilon_s$ in the preconditioner is large 
compared to $\kappa(\mr_s)\inv$, then the bound in Theorem~\ref{t_2} can be dominated by the least squares residual. Section~\ref{exp_accuracy} illustrates that this makes this bound uninformative in mixed-precision.

\subsection{Improved HPNE perturbation bound} \label{HPNE_new}
  We improve Theorem~\ref{t_2} with a bound that is informative in  both mixed precision and a single working precision.  

\begin{theorem}\label{t_22}
Let $\ma\in\rmn$ have $\rank(\ma)=n$;  $\mr_s\in\rnn$ be nonsingular; $\me_s\in\rnn$, $\me_A\in\rmn$; $\|\me_s\|\|\mr_s\inv\|<1$, and
\begin{align*}
\ma_1\equiv (\ma+\me_A)(\mr_s+\me_s)^{-1},\qquad \ma_2\equiv \ma+\me_A, \qquad
\epsilon_s\equiv\frac{\|\me_s\|}{\|\mr_s\|}, 
\qquad \epsilon_A\equiv\frac{\|\me_A\|}{\|\ma\|}.
\end{align*}
The computed solutions corresponding to (\ref{e_hpne}) are
\begin{align}
\ma_1^T\ma_2\vhx=\ma_1^T\vb.\label{e_t212}
\end{align}
If  $\vhx\neq \vzero$, then
\begin{align}\label{e_hpnetrue}
    \frac{\|\vhx-\vx_*\|}{\|\vhx\|}\leq\kappa(\ma_p^T\ma)\nu\epsilon_A\left(\kappa(\mr_s)\frac{\|\ma\vhx-\vb\|}{\|\ma\|\|\vhx\|}+1+\kappa(\ma)\epsilon_A\right),
\end{align}
where 
\begin{align*}
\nu\equiv \frac{\|\ma_p\|\|\ma\|}{\|\ma_p^T\ma\|}\geq 1.
\end{align*}
\end{theorem} 

\begin{proof}
As in the proof of Theorem~\ref{t_55}, the Banach lemma \cite[Lemma 2.3.3]{GovL13} implies that $\mr_s+\me_s=(\mi+\me_s\mr_s\inv)\mr_s$ is invertible, and that $\ma_1$ is well defined. Since $\ma$
has full column rank,
\begin{align*}
\ma_1=(\mi+\me_A\ma^\dagger)\ma_p(\mi+\me_s\mr_s\inv)\inv=(\mi+\mh)\ma_p(\mi+\me_s\mr_s\inv)\inv, \qquad \mh\equiv\me_A\ma^\dagger.
\end{align*}
Find the computed least squares residual in~\eqref{e_t212} as 
\begin{align*}
\ma_1^T\vr=-\ma_1^T\me_A\vhx\qquad \text{where} \qquad 
\vr\equiv\ma\vhx-\vb,
\end{align*}
multiply both sides by $(\mi+\me_s\mr_s\inv)^T$ and use the above
expression for $\ma_1$,
\begin{align*}
   \ma_p^T\vr=-\mr_s^{-T}\me_A^T\vr-\ma_p^T(\mi+\mh)^T\me_A\vhx.
\end{align*}
At last, multiply both sides by $(\ma_p^T\ma)\inv$ and take norms, 
\begin{align*}
    \frac{\|\vhx-\vx_*\|}{\|\vhx\|}&\leq\kappa(\ma_p^T\ma)\frac{\|\ma_p\|}{\|\ma_p^T\ma\|}\left(\kappa(\mr_s)\|\me_A\|\frac{\|\vr\|}{\|\ma_p\|\|\mr_s\|\|\vhx\|}+(1+\kappa(\ma)\epsilon_A)\|\me_A\|\right)
    \\
    &\leq\kappa(\ma_p^T\ma)\frac{\|\ma_p\|\|\ma\|}{\|\ma_p^T\ma\|}\left(\kappa(\mr_s)\epsilon_A\frac{\|\vr\|}{\|\ma\|\|\vhx\|}+(1+\kappa(\ma)\epsilon_A)\epsilon_A\right),
\end{align*}
where the last inequality follows from $\ma=\ma_p\mr_s$.
\end{proof}

Theorem~\ref{t_22} shows that to first order the HPNE error is bounded by 
\begin{equation}\label{e_fohpne}
    \frac{\|\vhx-\vx_*\|}{\|\vhx\|}\lesssim\kappa(\ma_p^T\ma)\nu\epsilon_A
    \max\left\{\kappa(\mr_s)\frac{\|\ma\vhx-\vb\|}{\|\ma\|\|\vhx\|}, \, 1\right\}.
\end{equation}
With an effective preconditioner, so that $\kappa(\ma_p)\lesssim10$, this bound resembles the one in Lemma~\ref{l_ls} since  $\kappa(\ma_p^T\ma)\approx\kappa(\ma)$, see~\cite[Section 3]{ipsen2025solutionsquaresproblemsrandomized} for details on this approximation. Similar to the PNE bound in Theorem~\ref{t_55}, the bound in Theorem~\ref{t_2} does not depend on the perturbation $\epsilon_s$ in the preconditioner, as long as $\mr_s+\me_s$ is nonsingular.

\section{Seminormal and not-normal equations }\label{s_var}
We consider two alternative approaches for normal equations,
and present perturbation bounds for the seminormal equations (Section~\ref{s_sne}) and the not-normal equations (Section~\ref{s_nne}).

\subsection{Seminormal equations}\label{s_sne}
If $\ma=\mq\mr$ is a thin QR factorization where $\mq\in\rmn$ has orthonormal columns and $\mr\in\rnn$ is upper triangular, the seminormal equations~\cite{BJORCK198731} are
 \begin{equation}\label{e_sne}
    \mr^T\mr\vx_*=\ma^T\vb.
\end{equation}
 The seminormal equations can be used for solving least squares problems~(\ref{e_ls}) when $\ma$ is large and sparse,
 or when $\mq$ is expensive to access or store~\cite{BJORCK198731}.
 
The perturbation bound in Theorem~\ref{l_semi} below makes no  assumptions on the perturbation $\me$, so the perturbed matrix $\ma+\me$ can be rank-deficient.

\begin{theorem}\label{l_semi} 
Let $\ma\in\rmn$ have $\rank(\ma)=n$; $\mr_s\in\rnn$ be nonsingular; $\me\in\rmn$, $\epsilon\equiv\|\me\|/\|\ma\|$; 
and $\ma+\me=\hat{\mq}\hat{\mr}$ a thin QR factorization where $\hat{\mq}\in\rmn$ has orthonormal columns.
The computed solutions corresponding to (\ref{e_sne}) are
\begin{equation}\label{eq1}
       \hat{\mr}^T\hat{\mr}\vhx=(\ma+\me)^T\vb.
    \end{equation}
If $\vhx\neq\mathbf{0}$ then 
    \begin{equation}
        \frac{\|\vhx-\vx_*\|}{\|\vhx||}\leq \kappa(\ma)^2\epsilon \left(\frac{\|\ma\vhx-\vb\|}{\|\ma\|\|\vhx\|}+1+\epsilon\right)
    \end{equation}
\end{theorem}

\begin{proof}
 Because $\hat{\mq}$ has orthonormal columns, we can write
 \begin{align*}
(\ma+\me)^T\vb=    \hat{\mr}^T\hat{\mr}\vhx
=    \hat{\mr}^T\hat{\mq}^T\hat{\mq}\hat{\mr}\vhx
= (\ma+\me)^T(\ma+\me)\vhx.   
 \end{align*}
Now apply Lemma~\ref{l_ne}.
\end{proof}

The bound in Theorem~\ref{l_semi} is the same as the one for the normal equations in Lemma~\ref{l_ne}. This is not surprising because $\mr^T\mr=\ma^T\ma$ in exact arithmetic.  The roundoff error analysis in \cite[Theorem 3.1]{BJORCK198731} also shows that the seminormal equations depend on $\kappa(\ma)^2$, though 
the perturbation analysis here assumes exact arithmetic after the 
perturbations to the inputs.

Theorem~\ref{l_semi} shows that, to first order, the relative error in $\vhx$ is bounded by
\begin{equation*}
\frac{\|\vx_*-\vhx\|}{\|\vhx\|}\lesssim \kappa(\ma)^2\,\epsilon\, 
\max\left\{\frac{\|\vb-\ma\vhx\|}{\|\ma\|\|\vhx\|}, 1\right\}.
\end{equation*}
Thus, like the normal equations, the conditioning of the seminormal equations
depends on the least squares residual when it is too large, that is, if
$\frac{\|\vb-\ma\vhx\|}{\|\ma\|\|\vhx\|}>1$. This quantifies the claim that the seminormal equations together with iterative refinement are insensitive to the size of the least squares residual in~\cite[Section 8]{carson2025}; the seminormal equations are sensitive to the size of the least squares residual, but only when it is large. However, much like the normal equations in Lemma~\ref{l_ls}, the error depends on the square of the condition number of $\ma$.

\subsection{Not-normal equations}\label{s_nne}
For a matrix  $\mb\in\rmn$ with $\rank(\mb)=n$ and  the same column space as $\ma$, the not-normal equations are~\cite{Wathennotnormal}
\begin{equation}\label{e_nne}
    \mb^T\ma\vx_*=\mb^T\vb.
\end{equation}
The not-normal equations, combined with an iterative linear system solver, represent a potential approach for solving least squares problems that are large and sparse \cite{Wathennotnormal}.

The perturbation bound in Theorem~\ref{l_notnorm} below makes no  assumptions on the perturbations $\me_A$ and $\me_B$, so the perturbed matrices $\ma+\me_A$ and $\mb+\me_B$ can be rank-deficient.

\begin{theorem}\label{l_notnorm}
 Let $\ma,\mb\in\rmn$ have $\rank(\ma)=\rank(\mb)=n$;
 $\mc\in\rnn$ be nonsingular with $\ma=\mb\mc$; $\me_A\in\rnn$, $\me_B\in\rmn$; $\epsilon_A\equiv\|\me_A\|/\|\ma\|$, and
$\epsilon_B\equiv\|\me_B\|/\|\mb\|$.
 The computed solutions corresponding to (\ref{e_nne}) are
\begin{equation}\label{eq2}
 (\mb+\me_B)^T(\ma+\me_A)\vhx=(\mb+\me_B)^T\vb.
 \end{equation}
If $\vhx\neq\vzero$ then
    \begin{equation}\label{eqnotnorm}
        \frac{\|\vhx-\vx_*\|}{\|\vhx||}\leq \kappa(\mb^T\ma)\nu\left(\epsilon_B\frac{\|\ma\vhx-\vb\|}{\|\ma\|\|\vhx\|}|+(1+\epsilon_B)\epsilon_A\right), \qquad 
\nu\equiv\frac{\|\mb\|\|\ma\|}{\|\mb^T\ma\|}\geq 1.
    \end{equation}
\end{theorem}

\begin{proof}
Rearrange \eqref{eq2} to find the computed least squares residual,
    \begin{equation*}
        (\mb+\me_B)^T\vr=-(\mb+\me_B)^T\me_A\vhx,
        \qquad \text{where}\qquad \vr\equiv \ma\vhx-\vb.
    \end{equation*}
Rearrange the left side and use the fact that $\mb^T\vb=\mb^T\ma\vx_*$,
\begin{equation}\label{eq3333}
        \mb^T\ma(\vhx-\vx_*)=-\me_B^T\vr-(\mb+\me_B)^T\me_A\vhx.
    \end{equation}
From $\ma=\mb\mc$ follows $\mb^T\ma=\mb^T\mb\mc$ where
$\mb^T\mb$ is symmetric positive definite since $\mb$ has full column rank. Thus $\mb^T\ma$, as the product of two nonsingular matrices $\mb^T\mb$ and $\mc$, is also nonsingular. 

Multiply~\eqref{eq3333} by $(\mb^T\ma)\inv$ on both sides and take norms,
\begin{equation*}
        \|\vhx-\vx_*\|\leq\|(\mb^T\ma)\inv\|\left(\|\me_B\|\|\vr\|+(\|\mb\|+\|\me_B\|)\|\me_A\|\|\vhx\|\right). 
    \end{equation*}
At last, factor out $\|\ma\|\|\mb\|$ and divide by $\|\vhx\|\neq 0$,
        \begin{equation*}
        \frac{\|\vhx-\vx_*\|}{\|\vhx\|}\leq\kappa(\mb^T\ma)\frac{\|\mb\|\|\ma\|}{\|\mb^T\ma\|}\left(\epsilon_B\frac{\|\vr\|}{\|\ma\|\|\vhx\|}+(1+\epsilon_B)\epsilon_A\right).
    \end{equation*}
\end{proof}

Theorem~\ref{l_notnorm} shows that, to first order, the relative error in $\vhx$ is bounded by
\begin{equation} \label{e_fonne}
        \frac{\|\vhx-\vx_*\|}{\|\vhx\|}\lesssim\kappa(\mb^T\ma)\left(\epsilon_B\frac{\|\ma\vhx-\vb\|}{\|\ma\|\|\vhx\|}+\epsilon_A\right).
    \end{equation}
The above bound suggests that, like the HPNE, the 
conditioniong of the not-normal equations depends on the least squares residual.

\begin{remark}[HPNE as a special case of the  not-normal equations]
If $\mb=\ma_p$ in (\ref{e_nne}) then HPNE is a special case of the not-normal equations. 

Though the two bounds are identical to first order,
Theorem~\ref{l_notnorm} does not directly reduce to Theorem~\ref{t_22}. 
The small difference comes from the inability of Theorem~\ref{l_notnorm} to account for the perturbation in the preconditioner $\mr_s$, since it assumes an additive perturbation in $\mb$.

To see this, let $\mb=\ma_p$ and $\me_B=\me_A\mr_s\inv$, so that  $\mb+\me_B=(\ma+\me_A)\mr_s\inv$. Inserting these matrices into~\eqref{eqnotnorm} gives 
 \begin{align}\label{e_almost}
         \frac{\|\vhx-\vx_*\|}{\|\vhx\|}\leq\kappa(\ma_p^T\ma)\frac{\|\ma_p\|\|\ma\|}{\|\ma_p^T\ma\|}\epsilon_A\left(\kappa(\mr_s)\frac{\|\ma\vhx-\vb\|}{\|\ma\|\|\vhx\|}+1+\kappa(\mr_s)\,\epsilon_A\right)
     \end{align}   
since $\epsilon_B=\frac{\|\me_A\mr_s\inv\|}{\|\ma_p\|}\leq\kappa(\mr_s)\epsilon_A$. If $\kappa(\ma)\approx\kappa(\mr_s)$, then \eqref{e_almost} and the HPNE bound~\eqref{e_hpnetrue} are almost the same. Moreover, the first order approximation of~\eqref{e_almost} is  
\begin{align}\label{e_almost2}
         \frac{\|\vhx-\vx_*\|}{\|\vhx\|}\lesssim\kappa(\ma_p^T\ma)\frac{\|\ma_p\|\|\ma\|}{\|\ma_p^T\ma\|}\epsilon_A
\left\{\kappa(\mr_s)\frac{\|\ma\vhx-\vb\|}{\|\ma\|\|\vhx\|},\,1\right\},
     \end{align}   
which is identical to the first order approximation~\eqref{e_fohpne} of the HPNE bound.
\end{remark}

\begin{remark}[Comparison with PNE and HPNE]\label{s_comp}
If $\mr_s$ and $\mb$ are effective preconditioners with 
$\kappa(\ma_p)\lesssim 10$ and  
$\kappa(\mb)\lesssim 10$, then 
the not-normal equations bound in Theorem~\ref{l_notnorm} resembles the 
PNE and HPNE bounds in Theorems \ref{t_55} and~\ref{t_22}, respectively.
From~\cite[Section 3]{ipsen2025solutionsquaresproblemsrandomized} 
follows 
\begin{align*}
\kappa(\mb^T\ma)\approx\kappa(\ma_p^T\ma)\approx \kappa(\ma)\approx\kappa(\mr_s).
\end{align*}
With these approximations, the first order bounds of the PNE~\eqref{e_pnefo}, HPNE~\eqref{e_fohpne} and not-normal equations~\eqref{e_fonne} are all approximately equal to
\begin{align*}
     \frac{\|\vx_*-\vhx\|}{\|\vhx\|}\lesssim
\kappa(\ma)\,\nu\,\epsilon_A\max\left\{\kappa(\ma)\frac{\|\ma\vhx-\vb\|}{\|\ma\|\|\vhx\|},1\right\}.
 \end{align*}
If also $\nu\approx 1$, then this bound resembles that of the original least squares problem in Lemma~\ref{l_ls}, suggesting that
all linear systems are well conditioned.
\end{remark}


\section{Preconditioning}\label{s_lpp}
We discuss the motivation behind the low-precision
preconditioner for PNE and HPNE (Section~\ref{mpprec}), and
present the randomized preconditioner for the numerical experiments (Section~\ref{rp}).

\subsection{A low precision preconditioner}\label{mpprec}
Theorems \ref{t_55} and~\ref{t_2}
suggest that the PNE and HPNE
conditioning depends only weakly on the perturbation in the preconditioner.
Hence we can accelerate the linear system solution by computing the preconditioner in lower precision. 

Lemma~\ref{rscond} shows that a
perturbation~$\me$ in the preconditioner has little effect on the condition number of the preconditioned matrix $\ma_p$, as long as $\me$ remains sufficiently small.

\begin{lemma}\label{rscond} 
Let $\ma\in\rmn$ with $\rank(\ma)=n$;
$\mr_s\in\rnn$ nonsingular; $\me\in\rnn$;
$\epsilon\equiv\frac{\|\me\|}{\|\mr_s\|}$;
$\|\me\|\|\mr_s^{-1}\|<1$; and  
\begin{align*}
\ma_p\equiv\ma\mr_s^{-1},\qquad \ma_1\equiv \ma(\mr_s+\me)\inv.
\end{align*}
Then
 \[\kappa(\ma_1)\leq\kappa(\ma_p)\>\frac{1+\epsilon\ \kappa(\mr_s)}{1-\epsilon\ \kappa(\mr_s)}. \]
\end{lemma}

\begin{proof} 
Since $\|\me\|\|\mr_s^{-1}\|<1$, the Banach lemma \cite[Lemma 2.3.3]{GovL13} implies that $\mr_s+\me$ is nonsingular,
so that $\ma_1$ is well defined. 
To relate $\ma_1$ to $\ma_p$, factor out $\mr_s^{-1}$,
\begin{align*}
\ma_1=\ma\mr_s^{-1}(\mi+\me\mr_s^{-1})^{-1}
=\ma_p(\mi+\me\mr_s^{-1})^{-1}.
\end{align*}
The Banach lemma \cite[Lemma 2.3.3]{GovL13} also implies 
\begin{align}\label{eqbound1}
\|\ma_1\|\leq 
\frac{\|\ma_p\|}{1-\|\me\|\|\mr_s^{-1}\|}
=\frac{\|\ma_p\|}{1-\epsilon\ \kappa(\mr_s)}.
\end{align}
From $\ma_1$ having full column rank follows that its
left inverse equals
\begin{align*}
\ma_1^{\dagger}=(\mi+\me\mr_s^{-1})\ma_p^{\dagger}.
\end{align*}
Take norms,
\begin{align*}
\|\ma_1^{\dagger}\|\leq 
\|\ma_p^{\dagger}\|\ (1+\|\me\|\|\mr_s^{-1}\|)
=\|\ma_p^{\dagger}\|\ (1+\epsilon\ \kappa(\mr_s)),
\end{align*}
and combine the above with \eqref{eqbound1}, 
\begin{align*}
\kappa(\ma_1)=\|\ma_1\|\|\ma_1^{\dagger}\|\leq 
\|\ma_p\|\|\ma_p^{\dagger}\|
\frac{1+\epsilon\ \kappa(\mr_s)}{1-\epsilon\ \kappa(\mr_s)}
=\kappa(\ma_p)\ \frac{1+\epsilon\ \kappa(\mr_s)}{1-\epsilon\ \kappa(\mr_s)}.
\end{align*}
\end{proof}

 Lemma~\ref{rscond} shows that if $\ma_p$ is well conditioned and if $\epsilon<\kappa(\mr_s)^{-1}$, then
 $\ma_1$ is also well conditioned. 
 The numerical experiments in Section~\ref{s_num} confirm
 that $\kappa(\ma_1)\lesssim 10$ as long as $\epsilon\ll \kappa(\mr_s)\inv \approx \kappa(\ma)\inv$. 
We control the size of $\epsilon$
 by adjusting the precision in which the preconditioner is computed.

 Algorithm~\ref{alg_5} presents a pseudo code for solving the PNE or HPNE with a preconditioner computed in lower precision.
The precision is automatically selected 
 by a fast condition number estimate 
 $\kappa_0\approx \log_{10}(\kappa(\ma))$
 computed in single precision. We use the Hager 1-norm condition number estimator $\verb|condest|$ in Matlab and the CONDITION package~\cite{burkardt2012condition} in C\texttt{++} to estimate the condition number of $\ma$. In particular, 
 we use the result from \cite[Exercise 2, Section 2.6]{IIbook} and estimate the 1 norm condition number of $\ma^T\ma$ to find the upper bound
 $$\kappa(\ma)^2=\kappa(\ma^T\ma)\leq n\|\ma^T\ma\|_1\|(\ma^T\ma)^{-1}\|_1.$$ 
 The selected precision is half precision if $\kappa_0<4$, and is single precision if $\kappa_0<8$. 
If $\kappa_0$ overflows or is greater than 8, the selected precision is double precision. If $\kappa(\ma)$ is known in advance, than the same precision selection heuristic applies without the need for estimation. 
 
  It is important to promote
 the preconditioner back to double precision prior to preconditioning~$\ma$.
 Experiments illustrate that the solution accuracy suffers if 
 the computation of $\ma\mr_s^{-1}$ happens in low precision.

\begin{algorithm}
\caption{PNE/HPNE with randomized low precision preconditioner}\label{alg_5}
\begin{algorithmic}
\Require $\ma\in\rmn$ with $\rank(\ma)=n$, $\vb\in\real^m$, sketching matrix $\mathbf{\Omega}\in\real^{d\times m}$ from~\eqref{e_sketch}. 
\Ensure Solution of PNE or HPNE  
\medskip
\State Estimate $\kappa_0\approx \log_{10}(\kappa(\ma))$ in single precision.\\
\medskip\qquad 
IF $\kappa_0<4, \vu_p:=\mathrm{half~precision}$\\
\medskip\qquad
ELSE IF $\kappa_0\leq 8, \vu_p:=\mathrm{single~precision}$\\
\medskip\qquad
ELSE  $\vu_p:=\mathrm{double~precision}$\\
\State Compute preconditioner in precision $\vu_p$.
\smallskip
\State{$\qquad$ Sketch $\ma_s= \mathbf{\Omega}\ma$} 
\smallskip
\State{$\qquad$ Thin QR factorization $\ma_s=\mq_s\mr_s$}
\smallskip
\State $\qquad$ Promote $\mr_s$ from precision $\vu_p$ to double precision. 
\medskip
\State Perform remaining computations in double precision.
\smallskip
\State{$\qquad$ Solve $\ma_p=\ma\mr_s^{-1}$} \qquad 
\smallskip
\State{$\qquad$ Solve PNE~\eqref{e_pne} or HPNE~\eqref{e_hpne}}
\end{algorithmic}
\end{algorithm}

\subsection{A randomized preconditioner}\label{rp}
The ideal preconditioner is the upper triangular matrix 
$\mr\in\rnn$ from a thin QR decomposition $\ma=\mq\mr$, where
$\mq\in\rmn$ has orthonormal columns. However, the computation of $\mr$ can be too expensive if $\ma$ has large dimension. 

A popular approach \cite{Blendenpik,carson2025,epperly2025fastrandomizedleastsquaressolvers,garrisonipsen} reduces the cost of computing a randomized preconditioner without a sacrifice in accuracy. 
This is done by sketching $\ma$ with a random matrix $\mathbf{\Omega}\in\real^{d\times m}$ 
to produce the smaller matrix $\mathbf{\Omega}\ma=\mq_s\mr_s$.
With an appropriate choice of $\mathbf{\Omega}$, 
the preconditioned matrix $\ma_p\equiv\ma\mr_s^{-1}$ is likely 
to have a low condition number.

Motivated by the least squares solver Blendenpik~\cite{Blendenpik},
the numerical experiments in Section~\ref{s_num} 
use a subsampled trigonometric transform as the sketching matrix,
\begin{align}\label{e_sketch}
    \mathbf{\Omega}\equiv\ms\mf\md,
\end{align}
where $\mf\in\mathbb{C}^{m\times m}$ is a Fourier Transform, 
$\md\in\rmm$ is a diagonal matrix whose diagonal elements 
are equal to -1 or 1 with equal probability 1/2, and the rows
of  $\ms\in\real^{d\times m}$ are sampled from $\mi_m$ uniformly and with replacement. 
The purpose of the randomized transform $\mf\md$ is to 
ensure that $\mf\md\ma$ has an optimal coherence $\mu\approx n/m$. 

Among the probabilistic bounds for the condition number of the preconditioned matrix based on (\ref{e_sketch}) \cite{Blendenpik,ipsen2025solutionsquaresproblemsrandomized,IpsW12,RT08}, we choose the following.

\begin{lemma}[Theorem 4.1 in \cite{ipsen2025solutionsquaresproblemsrandomized}]
Let $\ma\in\rmn$ with $\rank(\ma)=n$ and thin QR factorization $\ma=\mq\mr$, where $\mq\in\rmn$ with $\mq^T\mq=\mi_n$. Let $\ms$ in (\ref{e_sketch}) sample $d$ rows uniformly
and with replacement, and let $\mf\md\mq$ have coherence
$\mu\equiv \max_{1\leq i\leq m}{\|\ve_i^T\mf\md\mq\|_2^2}$.

For any $0<\epsilon<1$ and $0<\delta<1$, if $d\geq 2m\mu (1+\tfrac{\epsilon}{3})\tfrac{\ln(n/\delta)}{\epsilon^2}$
then with probability at least $1-\delta$
\begin{equation}\label{e_condAp} 
\kappa(\ma_p)\leq \sqrt{\frac{1+\epsilon}{1-\epsilon}}.
\end{equation}
\end{lemma}

Section~\ref{s_num} illustrates that a sampling amount of $d=3n$ tends to produce preconditioned matrices with $\kappa(\ma_p)\lesssim 10$.
 
\section{Numerical Experiments}\label{s_num}
After describing the setup of the  experiments (Section~\ref{s_expsetup2}), we present numerical experiments for the tightness of the PNE and HPNE perturbation bounds 
in double precision (Section~\ref{s_num1})
and in mixed precision (Section~\ref{exp_accuracy}); and the speed and accuracy of a low precision
preconditioner on an NVIDIA H100 GPU (Section~\ref{exp_HPC}).
To make randomized sampling effective, the 
 matrices $\ma\in\rmn$ are tall and skinny with $m\leq 2^{17}\approx10^5$ rows and $n\leq 2000$ columns.

\begin{algorithm}[!t]
\caption{Constructing exact least squares problems}\label{alg_exact}
\begin{algorithmic}
\Require Matrix dimensions $m$ and $n$, condition number $\kappa$
\State $\qquad\ $  Least squares residual norm $\rho$
\Ensure Matrix $\ma\in\rmn$ with $\kappa(\ma)=\kappa$,
right hand side $\vb\in\real^m$
\State $\qquad\quad$ Solution $\vx_*\in\rn$ with $\|\vx_*\|=1$
\State $\qquad\quad$ Least squares residual $\ve\equiv \vb-\ma\vx_*\in\real^m$
with $\|\ve\|=\rho$ 
\medskip

\State \Comment{Compute $\ma$}
\State Compute orthogonal matrix $\mq=\begin{bmatrix}\mq_1&\mq_2\end{bmatrix}\in\rmm$ with $\mq_1\in\rmn$
\State Compute upper triangular matrix $\mr\in\rnn$ with $\kappa(\mr)=\kappa$ and $\|\mr\|=1$
\State Multiply $\ma=\mq_1\mr\qquad$
\Comment{Thin QR with $\range(\mq_1)=\range(\ma)$}
\medskip

\State \Comment{Compute solution $\vx_*$ with $\|\vx_*\|=1$}
\State $\vx = \texttt{randn}(n, 1)\qquad$ \Comment{Standard random normal vector}
\State $\vx_*=\vx/\|\vx\|$
\medskip

\State \Comment{Compute least squares residual}
\State $\ve_r= \mq_2\mq_2^T\ \texttt{randn}(m,1)\qquad$
\Comment{noisevector $\ve_r$ orthogonal to $\range(\ma)$}
\State $\ve = \rho\,\ve_{r}/\|\ve_{r}\|\qquad$
 \Comment{Absolute residual  norm $\|\ma\vx_*-\vb\|=\rho$}
 \medskip
 
 \Comment{Compute righthand side $\vb$}
\State $\vb = \ma\vx_*+\ve\qquad $
\end{algorithmic}
\end{algorithm}

\subsection{Setup}\label{s_expsetup2}
Algorithm~\ref{alg_exact} presents Matlab pseudocode for 
constructing 'exact' least squares problems as motivated by \cite[Section 1.5]{MNTW24}. 
Since $\|\ma\|=\|\vx_*\|=1$ by construction, the absolute least squares residuals $\|\ma\vx_*-\vb\|$ are equal to the relative least squares residuals. 
The subsequent figures plot relative errors against least squares residuals whose norms vary from $10^{-16}$ to 1. 
We believe that these different least squares residuals are responsible for causing the slight fluctuations in the solution accuracy. 

The subsampled trigonometric transform $\mathbf{\Omega}$
in  (\ref{e_sketch}) samples $d=3n$ rows. 
Linear systems are solved with Matlab's \verb|mldivide| command.
We summarize the perturbation bounds below, and state them with general precisions $u_1$ and~$u_2$.
The old bounds are expressed in terms of
\begin{align*}
\eta_1 \equiv \left|\frac{\kappa(\mr_s)u_1}{1-\kappa(\mr_s)\,u_1}\right|,
\end{align*}
where the absolute value ensures that $\eta_1$ remains positive 
even if $\kappa(\mr_s)\,u_1>1$.

\begin{enumerate}
\item Old PNE bound in Theorem~\ref{t_2},
\begin{align}
     \frac{\|\vx_*-\vhx\|}{\|\vhx\|} \leq \kappa(\mr_s)\,\kappa(\ma_p)\frac{\|\mr_s\vhx\|}{\|\mr_s\| \|\vhx\|}
\left(u_2+\kappa(\ma_p)\, \eta_1
\left(\frac{\|\ma_p\vhy-\vb\|}{\|\ma_p\|\|\vhy\|}+u_2\right)\right). \label{eq31}
\end{align}
\item Improved PNE bound in Theorem~\ref{t_55}, 
\begin{align}
     \frac{\|\vx_*-\vhx\|}{\|\vhx\|} \leq \kappa(\mr_s)\kappa(\ma_p)u_2\left(\kappa(\ma_p)\kappa(\mr_s)\frac{\|\ma\vhx-\vb\|}{\|\ma\|\|\vhx\|}+ 1+\kappa(\ma)u_2 \right). \label{eq41}
\end{align}
\item Old HPNE bound in Theorem~\ref{t_2},
\begin{align}
     \frac{\|\vx_*-\vhx\|}{\|\vhx\|} \leq \kappa(\ma_p^T\ma)\, \frac{\|\ma_p\|\|\ma\|}{\|\ma_p^T\ma\|} 
\left(\eta_1\frac{\|\vb-\ma\vhx\|}{\|\ma\|\|\vhx\|}+(1+\eta_1)u_2\right). \label{eq42}
\end{align}
\item Improved HPNE bound in Theorem~\ref{t_22},
\begin{align}
 \frac{\|\vx_*-\vhx\|}{\|\vhx\|} \leq \kappa(\ma_p^T\ma)\frac{\|\ma_p\|\|\ma\|}{\|\ma_p^T\ma\|}  u_2\left(\kappa(\mr_s)\frac{\|\ma\vhx-\vb\|}{\|\ma\|\|\vhx\|}+1+\kappa(\ma)u_2\right).\label{eq32}
\end{align}
\end{enumerate}

\subsection{Double Precision}\label{s_num1}
We illustrate the tightness of the perturbation bounds 
and the accuracy for
PNE (Figure~\ref{fig:PNEacc1001}) and
HPNE (Figure~\ref{fig:HPNEacc1001}) in double precision. 
Since analogous results hold in single precision, we omit them here for brevity.

The matrices
$\ma$ have $m=6,000$ rows, $n=100$ columns, and condition number 
$\kappa(\ma)=10^4$. 
The PNE bounds \eqref{eq31} and~\eqref{eq41} in Figure~\ref{fig:PNEacc1001}, and the HPNE bounds \eqref{eq42} and~\eqref{eq32} in Figure~\ref{fig:HPNEacc1001} use the precision
$u_1=u_2=2^{-52}$.

\begin{figure}[htbp]  
 \centering
        \includegraphics[width=0.48\linewidth]{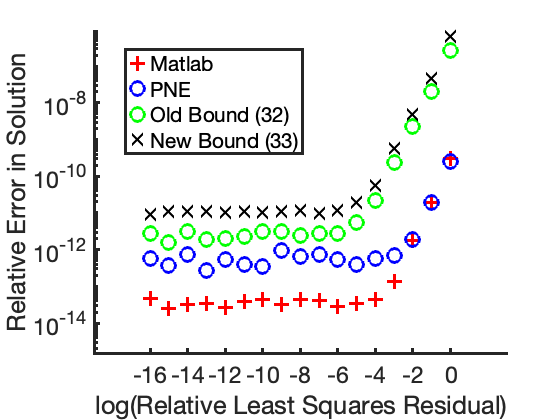}
       \caption{Relative errors in the computed solutions $\vhx$ and perturbation bounds versus logarithm of relative least squares residuals $\|\vb-\ma\vx_*\|/(\|\ma\|\|\vx_*\|)$ for matrices $\ma\in\real^{6,000\times 100}$ with condition number $\kappa(\ma)=10^4$. Shown are the errors in the Matlab backslash solutions (red plusses); the solutions from the PNE (blue squares); the  old bound~\eqref{eq31} (green circles) and the new bound~\eqref{eq41} (black x) with precision $u_1=u_2=2^{-52}$. }
\label{fig:PNEacc1001}
\end{figure}

\textbf{Figure~\ref{fig:PNEacc1001}.} The old PNE bound~\eqref{eq31} from Theorem~\ref{t_3} overestimates the error, but it is tighter than the new bound 
\eqref{eq41} from Theorem~\ref{t_55}. This is due
to the absence of the scaling factor $\frac{\|\mr_s\vhx\|}{\|\mr_s\| \|\vhx\|}\leq1$ 
in \eqref{eq41}. For a large enough least squares residuals,
$\|\ma\vx_*-\vb\|/(\|\ma\|\|\vx_*\|)>10 ^{-4}$, the PNE solution is as accurate as the one from Matlab's $\verb|mldivide|$.

\begin{figure}[htbp]
    \centering

        \centering
        \includegraphics[width=0.48\linewidth]{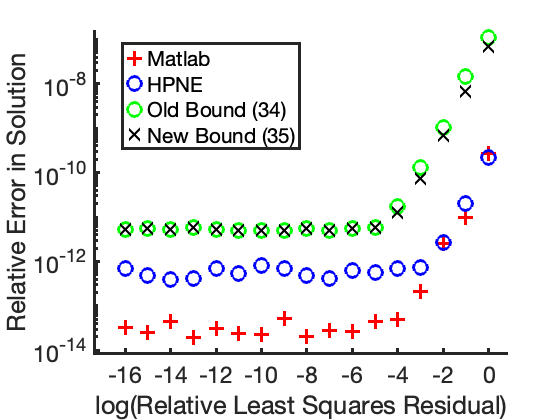}
      
    \caption{Relative errors in the computed solutions $\vhx$ and perturbation bounds versus logarithm of relative least squares residuals $\|\vb-\ma\vx_*\|/(\|\ma\|\|\vx_*\|)$ for matrices $\ma\in\real^{6,000\times 100}$ with condition number $\kappa(\ma)=10^4$. Shown are the errors in the Matlab backslash solutions (red plusses); the solutions from the HPNE (blue squares); the old bound~\eqref{eq42}  (green circles) and the new bound~\eqref{eq32} (black x)  with precision $u_1=u_2=2^{-52}$.}
     \label{fig:HPNEacc1001}
\end{figure}

\textbf{Figure~\ref{fig:HPNEacc1001}.}
The old bound~\eqref{eq42} from Theorem~\ref{t_2} and the new bound~\eqref{eq32} from Theorem~\ref{t_22} are essentially the same. 
For large enough least squares residuals,
$\|\ma\vx_*-\vb\|/(\|\ma\|\|\vx_*\|)>10^{-4}$, the HPNE 
solution is as accurate as the one from Matlab's $\verb|mldivide|$.

\subsection{Mixed Precision}\label{exp_accuracy}
We illustrate the tightness of the perturbation bounds 
and the accuracy for
PNE (Figure~\ref{fig:combined_PNE}) and
HPNE (Figure~\ref{fig:HPNE_combined}) in mixed precision.

The matrices $\ma$ have $m=6,000$ rows, and
$n=100$ or $n=1000$ columns. The condition number is $\kappa(\ma)=10^8$, where the normal equations fail in double precision while a single precision preconditioner can still be effective enough to achieve $\kappa(\ma_p)\leq 10$.   
The PNE and HPNE are solved with Algorithm~\ref{alg_5},
with 
precisions $u_1=2^{-23}$ and $u_2=2^{-52}$;  the same as
the PNE bounds \eqref{eq31} and~\eqref{eq41} in Figure~\ref{fig:combined_PNE}, and the HPNE bounds \eqref{eq42} and~\eqref{eq32} in Figure~\ref{fig:HPNE_combined}.

\textbf{Figure~\ref{fig:combined_PNE}.} 
The new bound~\eqref{eq41} from Theorem~\ref{t_55} 
reflects the qualitative behavior of the error, whereas the old bound~\eqref{eq31} from Theorem~\ref{t_3} is a severe overestimate.
For sufficiently small least squares residuals, 
$\|\ma\vx_*-\vb\|/(\|\ma\|\|\vx_*\|)<10^{-6}$
the relative error
in the PNE solution is 100 larger than that 
of  Matlab's $\verb|mldivide|$,
while for larger least squares residuals, the 
PNE solution is as accurate.

\begin{figure}[htbp]
    \centering
    \begin{subfigure}[b]{0.48\linewidth}
        \centering
        \includegraphics[width=\linewidth]{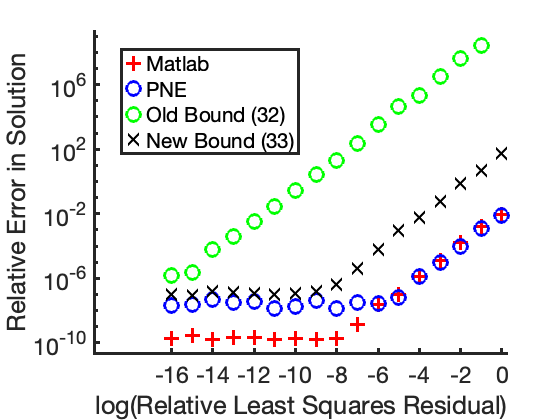}
        \caption{$\ma\in\real^{6,000\times 100}$}
        \label{fig:PNEacc100}
    \end{subfigure}
    \hfill 
    \begin{subfigure}[b]{0.48\linewidth}
        \centering
        \includegraphics[width=\linewidth]{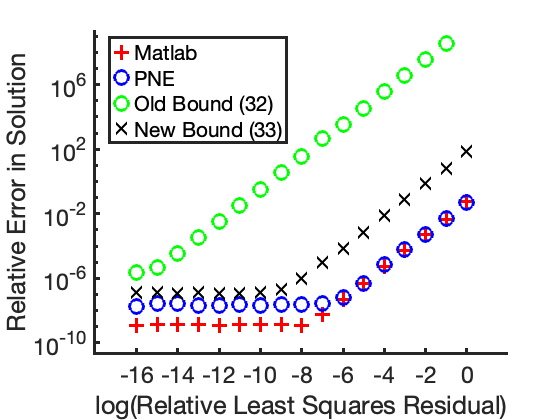}
        \caption{$\ma\in\real^{6,000\times 1000}$}
        \label{fig:PNEacc1000}
    \end{subfigure}

    \caption{Relative errors in the computed solutions $\vhx$ and perturbation bounds versus logarithm of relative least squares residuals $\|\vb-\ma\vx_*\|/(\|\ma\|\|\vx_*\|)$ for matrices with condition number $\kappa(\ma)=10^8$. Shown are the errors in the Matlab backslash solutions (red plusses); the solutions from the PNE with a single precision preconditioner (blue squares); the old bound~\eqref{eq31} (green circles) and the new bound~\eqref{eq41} (black x) with 
    precisions $u_1=2^{-23}$ and $u_2=2^{-52}$.}
    \label{fig:combined_PNE}
\end{figure}

\textbf{Figure~\ref{fig:HPNE_combined}.} 
 The new bound~\eqref{eq32} from Theorem~\ref{t_22} 
 reflects the qualitative behavior of the error, whereas the old bound~\eqref{eq42} from Theorem~\ref{t_2} 
 is a severe overestimate.
 For sufficiently small least squares residuals,
$\|\ma\vx_*-\vb\|/(\|\ma\|\|\vx_*\|)< 10^{-6}$, 
the relative error in the HPNE solution is 100 times larger
than that of Matlab's $\verb|mldivide|$, while for larger
least squares residuals, the HPNE solution is as accurate.

\begin{figure}[htbp]
    \centering
    \begin{subfigure}[b]{0.48\linewidth}
        \centering
        \includegraphics[width=\linewidth]{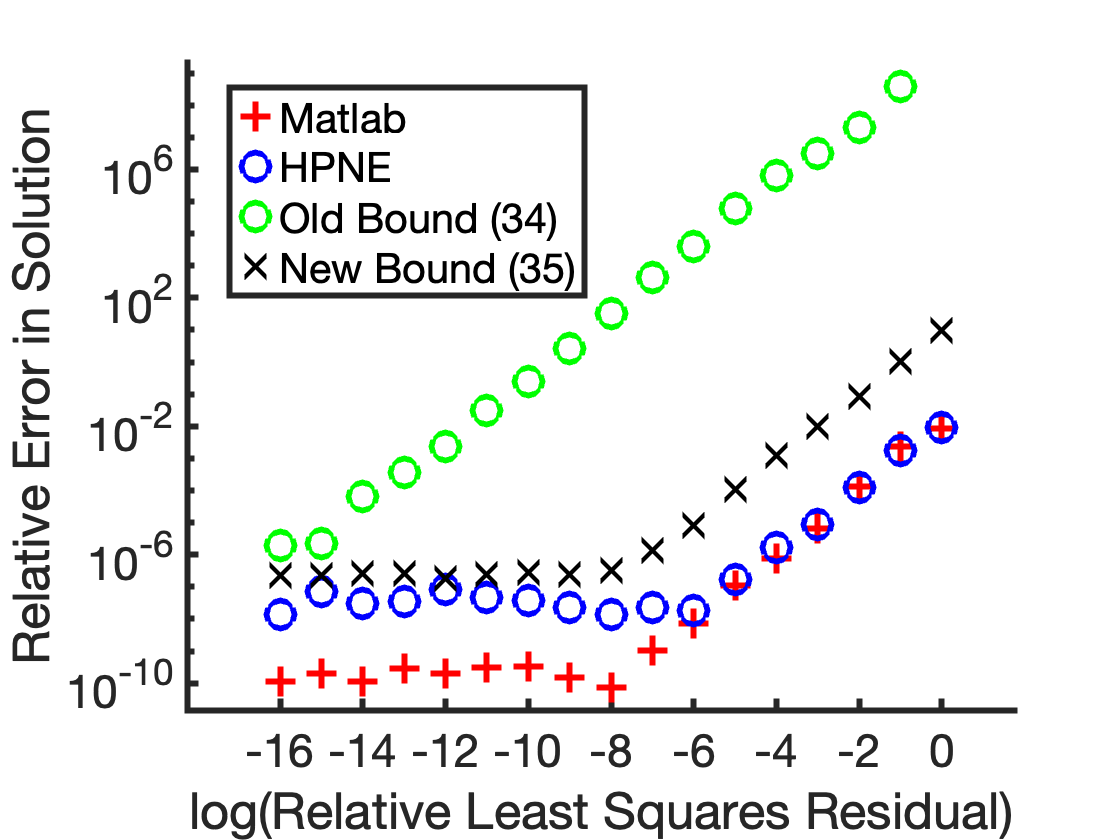}
        \caption{$\ma\in\real^{6,000\times 100}$}
        \label{fig:HPNEacc100}
    \end{subfigure}
    \hfill
    \begin{subfigure}[b]{0.48\linewidth}
        \centering
        \includegraphics[width=\linewidth]{HPNE1000.png}
        \caption{$\ma\in\real^{6,000\times 1000}$}
        \label{fig:HPNEacc1000}
    \end{subfigure}

    \caption{Relative errors in the computed solutions $\vhx$ and perturbation bounds versus logarithm of relative least squares residuals $\|\vb-\ma\vx_*\|/(\|\ma\|\|\vx_*\|)$ for matrices with condition number $\kappa(\ma)=10^8$. Shown are the errors in the Matlab backslash solutions (red plusses); the solutions from the HPNE with a single precision preconditioner (blue squares); the old bound~\eqref{eq42} (green circles) and the new bound~\eqref{eq32} (black x) with 
    precisions $u_1=u^{-23}$ and $u_2=u^{-52}$. }
    \label{fig:HPNE_combined}
\end{figure}

\subsection{Mixed Precision on a GPU}\label{exp_HPC}
We compare mixed precision PNE/HPNE with double precision PNE/HPNE on a GPU. We examine speedup in~Figure~\ref{fig:GPUcomparison2} and the accuracy in Figure~\ref{fig:GPUcomparisonacc2}.

The experiments  were performed in C$\texttt{++}$ on an NVIDIA H100 PCIe GPU with 80 GB of HBM3 memory, utilizing CUDA 12.6; vendor-optimized libraries such as cuBLAS, cuSOLVER, and cuFFT; and the open-source C\texttt{++} parallel algorithm library Thrust.
We compare the PNE and HPNE to a standard QR-based solver based on optimized \verb|geqrf| and \verb|orgqr| routines, and \verb|trsm| 
for solving triangular systems.
Wall clock times are measured with the CUDA\verb| gettimeofday| function.
 
Half precision floating point numbers are stored in
the CUDA \verb|__half| data type. Due to limitations
of cuFFT in half precision, the row dimensions $m$ of $\ma$ must be a power of two.

\begin{figure}[h!]
    \centering
    \begin{subfigure}[b]{0.48\linewidth}
        \centering
        \includegraphics[width=\linewidth]{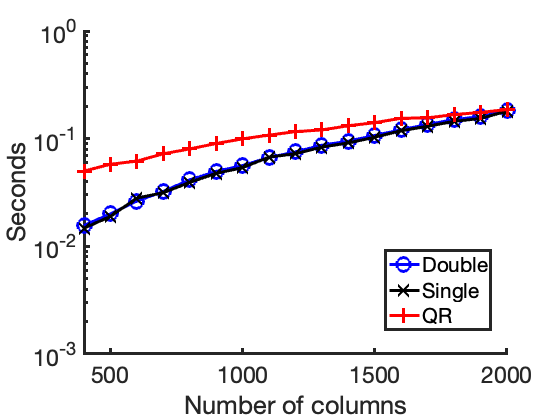}
        \caption{Average run times. }
        \label{fig:100000GPUsingraw}
    \end{subfigure}
    \hfill
    \begin{subfigure}[b]{0.48\linewidth}
        \centering
        \includegraphics[width=\linewidth]{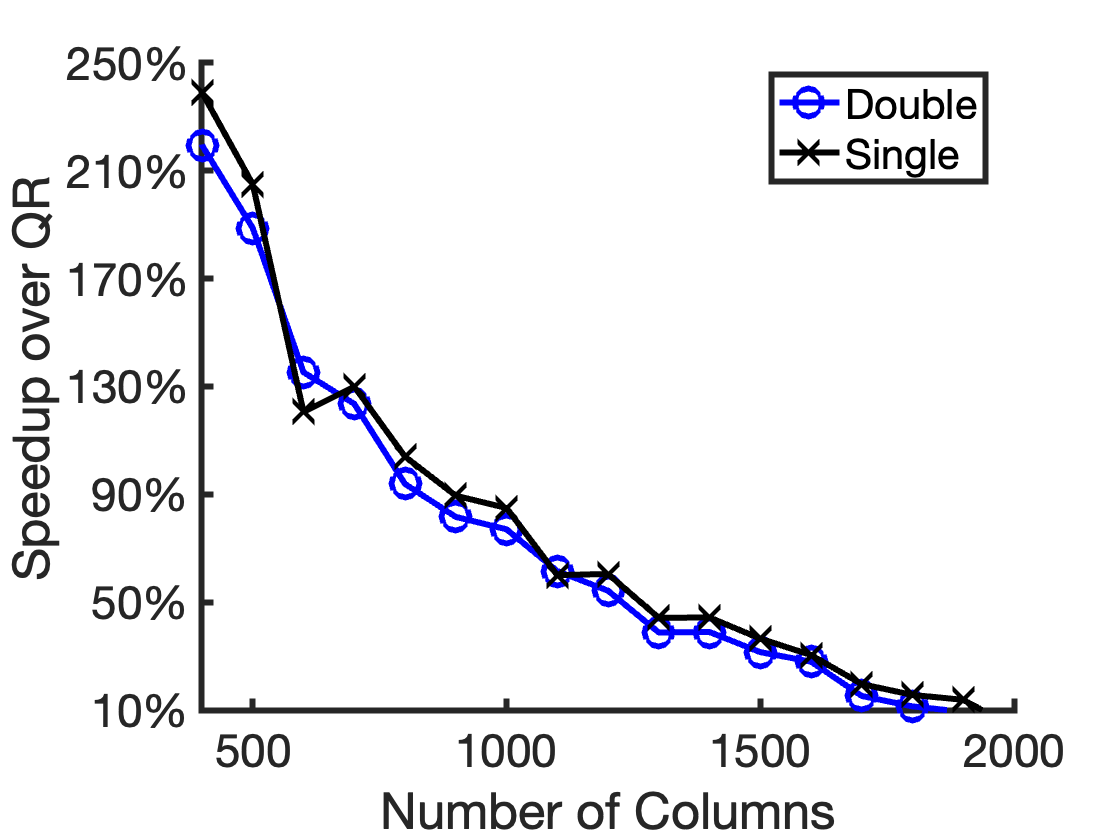}
        \caption{Average speedups over QR solver. }
        \label{fig:100000GPUspeedup}
    \end{subfigure}
    \caption{Average run times and relative speedups in seconds, over 10 trials, versus number of columns $400\leq n\leq 2000$
    for matrices $\ma\in\rmn$ with $m=10^5$ rows.  Shown are times for double precision PNE (blue circles), mixed precision PNE (black crosses), and QR solver (red plusses).}
    \label{fig:GPUcomparison2}
\end{figure}

\textbf{Figure~\ref{fig:GPUcomparison2}.} 
The matrices $\ma\in\rmn$ in Figure~\ref{fig:GPUcomparison2} have $m=10^5$ rows, $400\leq n\leq 2000$ columns, condition number $\kappa(\ma)=10^7$ and least squares residuals $\|\ma\vx_*-\vb\|/(\|\ma\|\|\vx_*\|)=10^{-6}$. Because the focus of this experiment is about the comparison between single and double precision, we do not estimate the condition number of $\ma$ and instead specify the precision of the preconditioner in advance. Figure~\ref{fig:100000GPUsingraw} shows the average runtime, in seconds, for double precision PNE, mixed precision PNE, and a QR based solver. Figure~\ref{fig:100000GPUspeedup} illustrates the speedup of double and mixed precision PNE, 
computed as the ratio of PNE time over QR solver time.

The mixed precision PNE can be faster than the double precision PNE, as illustrated by $n=400, 500,$ and $1000$ columns in Figure~\ref{fig:100000GPUspeedup}.  However, because linear solves dominate the operation count of the PNE over the preconditioner computation, the increase in speed is only modest. Nevertheless, the mixed precision approach offers the complementary benefit of a reduced memory footprint. For $n\leq 2000$ columns, mixed precision and double precision PNE are faster than the QR solver.
As the number of columns grows in Figure~\ref{fig:100000GPUsingraw}, the sketching process becomes more expensive and all three algorithms have similar speed. 

\textbf{Figure~\ref{fig:GPUcomparisonacc2}.} 
The matrices $\ma\in\rmn$ in Figure~\ref{fig:GPUcomparisonacc2} have $m=2^{17}$ rows, $400\leq n\leq 2000$ columns, and least squares residuals $\|\ma\vx_*-\vb\|/(\|\ma\|\|\vx_*\|)=10^{-6}$. In Figure \ref{fig:GPUacchlaf}, $\kappa(\ma)=100$ and in Figure~\ref{fig:GPUaccdouble}, $ \kappa(\ma)=10^{10}$. Figure~\ref{fig:GPUcomparisonacc2} shows that the automatic precision selection is accurate for matrices of varying conditioning levels. In Figure~\ref{fig:GPUacchlaf}, the half precision preconditioner is almost as accurate as the QR solver for a well conditioned matrix, and the double precision preconditioner in~\ref{fig:GPUaccdouble} is as accurate as the QR solver for an ill conditioned matrix.  

\begin{figure}[h!]
    \centering
    \begin{subfigure}[b]{0.48\linewidth}
        \centering
        \includegraphics[width=\linewidth]{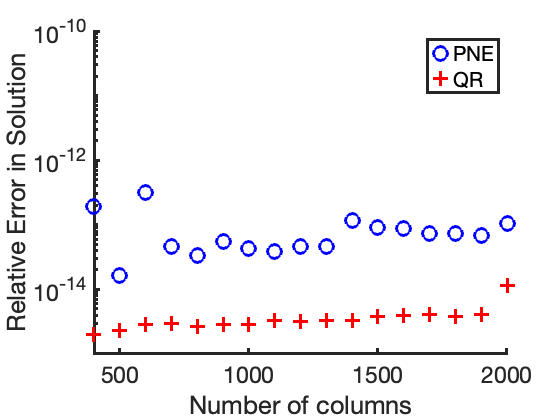}
        \caption{$\kappa(\ma)=10^2$ }
        \label{fig:GPUacchlaf}
    \end{subfigure}
    \hfill
    \begin{subfigure}[b]{0.48\linewidth}
        \centering
        \includegraphics[width=\linewidth]{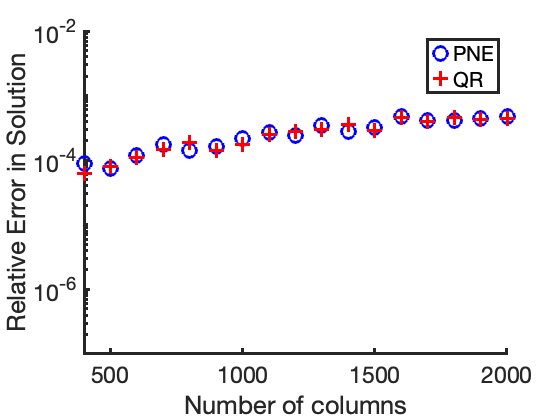}
        \caption{$\kappa(\ma)=10^{10}$ }
        \label{fig:GPUaccdouble}
    \end{subfigure}
    \caption{Relative errors in the computed solutions $\vhx$ and perturbation bounds versus number of columns $n$ for matrices with condition number $\kappa(\ma)=10^2$ (Figure~\ref{fig:GPUacchlaf}) and  $\kappa(\ma)=10^{10}$ (Figure~\ref{fig:GPUaccdouble}). Shown are the errors in the QR solver (red plusses); the errors from the PNE with an automatic selector for the precision of the preconditoner (blue circles).}
    \label{fig:GPUcomparisonacc2}
\end{figure}


\end{document}